\documentclass[12pt,a4paper,reqno]{amsart}
\allowdisplaybreaks
\usepackage{amsmath}
\usepackage{amsfonts}
\usepackage{amssymb,amsthm,latexsym,enumerate,url,cases}
\usepackage{booktabs,array,longtable}
\numberwithin{equation}{section}
\usepackage{mathrsfs}
\usepackage{hyperref}
\hypersetup{colorlinks=true,citecolor=blue,linkcolor=blue,urlcolor=blue}
\newcommand{\pmodtight}[1]{\,({\rm mod}\ #1)}

\theoremstyle{plain}
\newtheorem{theorem}{Theorem}[section]
\newtheorem{lemma}[theorem]{Lemma}

\newtheorem{proposition}[theorem]{Proposition}

\theoremstyle{definition}

\usepackage{etoolbox}
\makeatletter
\patchcmd{\@settitle}{\uppercasenonmath\@title}{}{}{}
\patchcmd{\@setauthors}{\MakeUppercase}{}{}{}
\patchcmd{\section}{\scshape}{}{}{}
\makeatother

\begin{document}

\title
[{An improved upper bound on the Ruzsa number}]
{An improved upper bound on the Ruzsa number}

\author
[Y. Ding, Y. Sun and L. Zhao] 
{Yuchen Ding, \quad Yu-Chen Sun, \quad Lilu Zhao}

\address{(Yuchen Ding) School of Mathematical Sciences,  Yangzhou University, Yangzhou 225002, People's Republic of China;\ 
HUN-REN Alfr\'ed R\'enyi Institute of Mathematics, Budapest, Pf. 127, H-1364 Hungary}
\email{ycding@yzu.edu.cn}
\address{(Yu-Chen Sun) School of Mathematics, University of Bristol
Bristol, BS8 1UG, England}
\email{yuchensun93@163.com}
\address{(Lilu Zhao) School of Mathematical Science, University of Science and Technology of China, Hefei 250100, People's Republic of China}
\email{zhaolilu@ustc.edu.cn}

\keywords{additive basis, Ruzsa number, Erd\H os--Tur\'{a}n conjecture}
\subjclass[2020]{11B13, 11B34.}

\begin{abstract}
Let $R_m$ be the least positive integer $r$ such that there exists a set $A\subseteq \mathbb{Z}_{m}$ with $A+A=\mathbb{Z}_m$ for which the number of ordered solutions of $n=x+y$ with $x,y\in A$ is at most $r$ for every $n\in \mathbb{Z}_m$.
In this note we prove that $R_m\leqslant 128$ for every positive integer $m$, improving the previous bound $R_m\leqslant 192$.
\end{abstract}
\maketitle

\section{Introduction}
For any positive integer $m$, let $\mathbb{Z}_{m}:=\mathbb{Z}/m\mathbb{Z}$ be the set of residue classes modulo $m$. For any non-empty subsets $A,B\subseteq \mathbb{Z}_{m}$ and any $n\in \mathbb{Z}_{m}$, let $\sigma_{A,B}(n)$ be the number of ordered solutions to the equation $n=x+y$ with $x\in A$ and $y\in B$. Let $\sigma_{A}(n)=\sigma_{A,A}(n)$.
The Ruzsa number $R_m$ is defined to be the least positive integer $r$ so that there exists a set $A\subseteq \mathbb{Z}_{m}$ with
$$1\leqslant \sigma_{A}(n)\leqslant r \quad (~\forall n\in \mathbb{Z}_{m}).$$
Equivalently, one asks for an additive basis of the cyclic group $\mathbb Z_m$ whose ordered representation function is bounded uniformly in the residue class. Thus $R_m$ is a finite combinatorial covering parameter, measuring whether the set $A$ can cover every element by sums while keeping the multiplicities uniformly bounded.

A subset $A$ of the natural numbers is called an asymptotic basis if 
$$f_{A}(n):=|\{(a,b)\in A^2:\ a+b=n\}|\geqslant 1$$
 for all sufficiently large integers $n$.
Erd\H os and Tur\'{a}n \cite{Er-Tu} posed a well-known conjecture which states that if $A$ is an asymptotic basis, then the representation function $f_{A}(n)$ cannot be bounded. Motivated by the above Erd\H os--Tur\'{a}n conjecture on the asymptotic basis,  Ruzsa \cite{Ruzsa} constructed an asymptotic basis $A$ of natural numbers which has a bounded square mean value, i.e., 
\begin{align}\label{R-bound}\sum_{n\leqslant N}f_{A}^2(n)\ll N.\end{align}
In the proof of the above result, Ruzsa essentially proved that there exists a constant $C$ such that 
$R_m\leqslant C$ for any positive integer $m$. 
Tang and Chen \cite{Tang1} proved an explicit bound that $R_m\leqslant 768$ for all sufficiently large $m$. Shortly after, they further \cite{Tang2} showed that $R_m\leqslant 5120$ for any $m$. The bound of $R_m$ was considerably improved by Chen \cite{Chen1}, who established $R_m\leqslant 288$ for any positive integer $m$. Chen's bound was recently improved to $R_m\leqslant 192$ by Ding and Zhao \cite{DZ24}.

The purpose of this paper is to provide a new upper bound on $R_m$. 

\begin{theorem}\label{thm1}
For any positive integer $m$, we have $R_m\leqslant 128$.
\end{theorem}

The main point of the proof is a new finite construction. In previous works, a key step was to obtain a bound for $R_{2p^2}$, where the letter $p$ denotes a prime number throughout. Indeed, both the bound $R_m\leqslant 288$ of Chen and the bound $R_m\leqslant 192$ of Ding and Zhao are based on the following result due to Chen \cite{Chen1}.
\begin{theorem}
[{\cite[Theorem 1]{Chen1}}]\label{thmChen}
Let $p$ be a prime. Then $R_{2p^2}\leqslant 48$.
\end{theorem}
Instead of investigating $R_{2p^2}$, the new ingredient in this paper is to work modulo $8p^2$.
In \cite{Tang3}, Tang and Chen gave some explicit upper bounds of  $R_{kp^2}$, where $1\leqslant k\leqslant 10$. In particular, they proved $R_{8p^2}\leqslant 64$.
 We prove the following stronger estimate of $R_{8p^2}$. 
\begin{theorem}\label{thm2}
Let $p$ be a prime. Then $R_{8p^2}\leqslant 32$.
\end{theorem}

The construction behind Theorem \ref{thm2} has two layers. The first layer consists of three quadratic graphs $\mathcal Q_{k_i}\subseteq \mathbb Z_p^2$, chosen so that Chen's quadratic-residue dichotomy guarantees coverage. The second layer consists of three small integer sets $E_1,E_2,E_3$ whose pairwise sums fill the interval $\{0,1,\ldots,8\}$. This additional residue layer is what allows us to pass from a construction modulo $2p^2$ to one modulo $8p^2$ while reducing the local multiplicity bound from $48$ to $32$. Finally, a comparison lemma of Ding and Zhao transfers this local estimate to arbitrary moduli and gives Theorem \ref{thm1}.

S\'{a}ndor and Yang \cite{SandorYang} gave the nontrivial lower bound of $R_m$. In particular, it is proved in \cite{SandorYang} that $R_m\geqslant 6$ for all sufficiently large $m$. For a different proof of the mean value estimate \eqref{R-bound}, one may refer to Chen and Yang \cite{ChenYang}.

Concerning the topic of the asymptotic basis in natural numbers, Grekos, Haddad, Helou and Pihko \cite{GHHP} showed that for an asymptotic basis $A$, the representation function $f_{A}(n)\geqslant 6$ for infinitely many $n$. It was then improved to $f_{A}(n)\geqslant 8$ infinitely often by Borwein, Choi and Chu \cite{BCC}. 
Chen \cite{Chen2} proved the existence of an asymptotic basis $A$ in the natural numbers which satisfies that the set of $n$ with $f_{A}(n)=2$ has density $1$. For more information on the Erd\H os--Tur\'{a}n conjecture, one can refer to Halberstam and Roth \cite{HalRo}, see also Tao and Vu \cite{Tao-Vu}.

We shall prove Theorem \ref{thm1} by using Theorem \ref{thm2} in Section 2 and prove Theorem \ref{thm2} in Section 3. 

\section{Proof of Theorem \ref{thm1} based on Theorem \ref{thm2}}
Before presenting the proof of Theorem \ref{thm1}, we prepare some lemmas.

We introduce a finite sequence as follows
$$
  c_j=74j-\left\lfloor \frac{j^2}{32}\right\rfloor\qquad (0\leqslant j\leqslant 264),
$$
where $\lfloor x\rfloor$ means the greatest integer no more than $x$. 
We first list some simple facts about the finite sequence $c_j$.  
Note that 
\begin{align}\label{cmax}
  c_{264}=74\times 264-\left\lfloor \frac{264^2}{32}\right\rfloor
  =19536-2178=17358>17350=132^2-74.
\end{align}
We record the elementary gap estimates
\begin{align}\label{gap}
  57\leqslant c_{j+1}-c_j\leqslant 74\ \textrm{ for }\ 0\leqslant j\leqslant 263.
\end{align}
The following finite fact is elementary. For reproducibility, the short verification code is included in Appendix~\ref{app:verification}.
\begin{lemma}\label{lem1}One has
$$
  \max_{n} \Big|\big\{(i,j):0\leqslant i,j\leqslant 264,
  \ c_i+c_j=n\big\}\Big|=17. 
$$
\end{lemma} 
Now we are ready to deal with $R_m$ for small $m$. 
\begin{lemma}\label{lem2}
For $m\leqslant 132^2$, we have {$R_m\leqslant 116$}.
\end{lemma}
\begin{proof}
For $m\leqslant 400$, let 
$$
A=\{0,1,\ldots,19\}\cup\{20,2\times 20,\ldots,19\times 20\}\pmodtight{m}.
$$
 It is easy to see that $A+A=\mathbb{Z}_{m}$ and $R_m\leqslant |A|\leqslant 40$.

Now, we assume that $400<m\leqslant 132^2$. Let $\beta=\beta_m$ be the smallest positive integer such that
\begin{align}\label{definebeta}
  c_{\beta}\geqslant m-74.
\end{align}
In view of \eqref{cmax}, we have $\beta\leqslant 264$. Moreover, we have $c_{\beta-1}<m-74$ and thus by \eqref{gap}, 
$$
c_{\beta}\leqslant c_{\beta-1}+74<m.
$$
 In particular, 
$c_1,\ldots,c_\beta$ is a sequence contained in $\{74,\ldots, m-1\}$. Now define
$$
  B=\{0,1,\ldots,73\},\qquad C=\{c_1,\ldots,c_{\beta}\}, \qquad \textrm{and }  \quad  A=B\cup C.
$$
Note that $A\subseteq \{0,1,\ldots,m-1\}$ and we can view $A$ as a subset of $\mathbb{Z}_m$ in the standard way. And we shall not distinguish the notation $A$ and $A\pmodtight{m}$. 
We wish to show that $1\leqslant \sigma_{A}(n)\leqslant 116$ for all $n\in \mathbb{Z}_m$. Due to \eqref{gap} and \eqref{definebeta}, by letting $C_0=C\cup\{0\}$, we have
\begin{align}\label{eq-new-1}
B+C_0\supseteq \{0,1,2,\ldots,m-1\}.
\end{align}
In fact, for $0\leqslant n\leqslant 73$ we have 
$$
n=n+0\in B+C_0.
$$
 For $74\leqslant n<m-75$, there is some $j$ such that
$
c_j\leqslant n<c_{j+1}.
$
Then $0\leqslant n-c_j<c_{j+1}-c_j\leqslant 74$. Hence, $n-c_j\leqslant B$ and 
$$
n=(n-c_j)+c_j\in B+C\subset B+C_0. 
$$
For $m-75\leqslant n\leqslant c_\beta-1$, we have
$$
n=c_{\beta-1}+(n-c_{\beta-1})\in C+B\subset B+C_0.
$$
For $c_\beta\leqslant n\leqslant m-1$, we have
$$
n=c_\beta+(n-c_\beta)\in C+B\subset B+C_0,
$$
completing the proof of \eqref{eq-new-1}.
In particular, from \eqref{eq-new-1} we get $A+A=\mathbb{Z}_m$. 

Let $0\leqslant n<m$. Note that 
\begin{align}\label{boundsigmaA}\sigma_{A}(n)\leqslant \sigma_{B}(n)+2\sigma_{B,C}(n)+\sigma_{C}(n),
\end{align}
where in the above inequality, we consider the additive problem over $\mathbb{Z}_m$. 

Firstly, we have $\sigma_{B}(n)\leqslant |B|= 74$. Now we deal with $\sigma_{B,C}(n)$. 
If $n\equiv x+y\pmodtight{m}$ with $x\in B$ and $y\in C$, then either $n=x+y$ or $n+m=x+y$. Let us consider $n'=x+y \in \mathbb{N}$. Then $y\in C\cap[n'-73,n']$. By the gap inequality \eqref{gap}, we have $|C\cap[n'-73,n']|\leqslant 2$. Thus the number of solutions to $n'= x+y$ in $\mathbb{N}$ with $x\in B$ and $y\in C$ is at most $2$. Since as discussed above, $n'$ can be $n$ or $n+m$, we obtain $\sigma_{B,C}(n)\leqslant 4$. Similarly, if $n\equiv x+y\pmodtight{m}$ with $x,y\in C$, then we also have either $n=x+y$ or $n+m=x+y$. In view of Lemma \ref{lem1}, we conclude that 
$\sigma_{C}(n)\leqslant 2\times 17=34$. Now recalling \eqref{boundsigmaA}, we deduce that 
$$\sigma_{A}(n)\leqslant 74+2\times 4+34=116.$$
The proof of the lemma is complete.
\end{proof}

\begin{lemma}[{\cite[Lemma 2.3]{DZ24}}]\label{lem3}
For any real number $x\geqslant 33$, there exists at least one prime in the interval $\left(x,\frac{2x}{\sqrt{3}}\right]$.
\end{lemma}

\begin{proposition}
[{\cite[Lemma 2.4]{DZ24}}]\label{prop}
Let $m_1$ and $m_2$ be two positive integers with $\frac{3}{2}m_1\leqslant m_2< 2m_1$. 
Then $$R_{m_2}\leqslant 4R_{m_1}.$$
\end{proposition}

\begin{proof}[Proof of Theorem \ref{thm1} by assuming Theorem \ref{thm2}]
By Lemma \ref{lem2}, we only need to consider $m> 132^2$.
By this condition we shall have $\sqrt{m}/4> 33$. By Lemma \ref{lem3}, there exists a prime $p$ such that
$$\frac{\sqrt{m}}{4}< p\leqslant \frac{2}{\sqrt{3}}\times \frac{\sqrt{m}}{4}.$$
In other words, we have $$\frac{3}{2}\times 8p^2\leqslant m< 2\times 8p^2.$$ Therefore, it follows from Theorem \ref{thm2} and Proposition \ref{prop} that
$$R_{m}\leqslant 4R_{8p^2}\leqslant 128.$$

This completes the proof of Theorem \ref{thm1}.
\end{proof}

\section{Proof of Theorem \ref{thm2}}

For an integer $k$, we introduce 
\begin{align*}\mathcal{Q}_k=\big\{(u,ku^2)\in \mathbb{Z}_p^2:\ u\in \mathbb{Z}_p\big\}.\end{align*}
\begin{lemma}[{\cite[Lemma 2.1]{Ruzsa}}]\label{lemma31}Let $k,\ell$ be two integers and $p$ be a prime satisfying $p\nmid k+\ell$. Let $(c,d)\in \mathbb{Z}_p^2$. Let $r(c,d)$ denote the number of solutions to the following equation
\begin{align*}(c,d)=(u_1,v_1)+(u_2,v_2) \ \textrm{ with }\ (u_1,v_1)\in \mathcal{Q}_{k} \ \textrm{ and } \ (u_2,v_2)\in \mathcal{Q}_{\ell}.\end{align*}
Let 
\begin{align*}\Delta(c,d)=(k+\ell)d-k\ell c^2.\end{align*}
Then we have
\begin{align*}r(c,d)=\begin{cases}2  \ & \textrm{ if }\ \Big(\frac{\Delta(c,d)}{p}\Big)=1
\\ 1  \ & \textrm{ if }\ \Big(\frac{\Delta(c,d)}{p}\Big)=0
\\ 0  \ & \textrm{ if }\ \Big(\frac{\Delta(c,d)}{p}\Big)=-1 \end{cases},\end{align*}
where $\Big(\frac{\Delta}{p}\Big)$ is the Legendre symbol.
\end{lemma}

Throughout this section, for an odd prime $p\geqslant 7$, we shall choose $m$ to be a quadratic non-residue modulo $p$ with
\begin{align*}m+1\not\equiv 0\pmodtight{p},\ 3m+1\not\equiv 0\pmodtight{p}\ \textrm{and} \ m+3\not\equiv 0\pmodtight{p}.\end{align*}
When $p>7$, the existence of such $m$ is clear since the number of quadratic non-residues modulo $p$ is $\frac{p-1}{2}$ which is greater than $3$. When $p=7$, we can choose $m=3$ (or $m=5$). We remark that when $p=5$, there is no $m$ satisfying the above four conditions.
And throughout, we fix
\begin{align*}k_1=m+1,\ k_2=m(m+1),\ k_3=2m.\end{align*}
Indeed,
\begin{align*}
 k_1+k_1&=2(m+1),& k_2+k_2&=2m(m+1),& k_3+k_3&=4m,\\
 k_1+k_2&=(m+1)^2,& k_1+k_3&=3m+1,& k_2+k_3&=m(m+3).
\end{align*}
The choice of $m$ makes all these quantities non-zero modulo $p$. Hence
\begin{align*}p\nmid k_i+k_j \ \textrm{ for all }\ 1\leqslant i,j\leqslant 3.\end{align*}
Now we renew the notations in Lemma \ref{lemma31} to indicate the dependence on $k_i$ and $k_j$.
Let $r_{ij}(c,d)$ denote the number of solutions to the following equation
\begin{align*}(c,d)=(u_1,v_1)+(u_2,v_2) \ \textrm{ with }\ (u_1,v_1)\in \mathcal{Q}_{k_i} \ \textrm{ and } \ (u_2,v_2)\in \mathcal{Q}_{k_j}.\end{align*}
and let 
\begin{align*}\Delta_{ij}(c,d)=(k_i+k_j)d-k_ik_jc^2.\end{align*}
\begin{lemma}[{Chen \cite{Chen1}}]\label{lemma32}For any $(c,d)\in \mathbb{Z}_p^2$, we have either $r_{12}(c,d)\geqslant 1$ or $r_{33}(c,d)\geqslant 1$. More precisely, by setting $\delta=d-mc^2$, we have
$$
\begin{array}{|c|c|c|c|}\hline
 & r_{12}(c,d)& r_{21}(c,d)& r_{33}(c,d)\\
\hline
 \delta\text{ is a quadratic residue} &2&2&0\\ 
 \hline
 \delta\text{ is a quadratic non-residue} &0&0&2\\ 
 \hline
 \delta\equiv 0\pmodtight{p}&1&1&1 \\
 \hline
\end{array}
$$
\end{lemma}
\begin{proof}The conclusion is contained in the proof of \cite[Lemma 2]{Chen1}.\end{proof}
Let us introduce 
\begin{align}\label{defineE}
  E_1=\{-3,-2,-1\},
  \qquad E_2=\{3,6,9\},
  \qquad E_3=\{0,1,3,4\}.
\end{align}
Directly checking, we easily observe that
\begin{align}\label{Ecomplete}E_1+E_2=E_3+E_3=\{0,1,2,3,4,5,6,7,8\}.\end{align}
We define $A_i\subseteq \mathbb{Z}_{8p^2}$ as
\begin{align}\label{defineAi} A_i=\big\{u+pe+8pv:(u,v)\in \mathcal{Q}_{k_i},\ e\in E_i\big\},
\end{align}
where we choose $0\leqslant u,v\leqslant p-1$. Note that in the above definition, $u+pe+8pv$ is understood modulo $8p^2$. For each fixed $i$, the representation is unique, because reducing modulo $p$ determines $u$, the elements of $E_i$ are pairwise distinct modulo $8$, and $v$ is then determined modulo $p$.
Then we define
\begin{align}A=\bigcup_{i=1}^3 A_i.\label{defineA}
\end{align}
 In fact, we shall prove $1\leqslant \sigma_A(n)\leqslant 32$ for all $n\in \mathbb{Z}_{8p^2}$. For this purpose, we introduce the following lemma.
\begin{lemma}\label{lem-rep}
For any integer $n$ with $0\leqslant n<8p^2$, there exists a unique representation that 
\begin{align*}
  n\equiv c+\varepsilon  p+8pd\pmodtight{8p^2},
  \qquad 0\leqslant c<p,
  \qquad 1\leqslant \varepsilon \leqslant 8, \qquad 0\leqslant d<p,
\end{align*}
where $c,d,\varepsilon $ are all integers.
\end{lemma}
\begin{proof}
There exists a pair $t~(p\leqslant t\leqslant 9p-1)$ and $d~(0\leqslant d<p)$ such that 
$$n\equiv t+8pd\pmodtight{8p^2}.$$ 
Actually, let $[n-p]$ be the least non-negative integer of $n-p$ mod $8p^2$. Then, by the division algorithm, there are two integers $0\leqslant d\leqslant p-1$ and $0\leqslant r\leqslant 8p-1$ such that
$
[n-p]=8pd+r.
$
Hence,
$$
n-p\equiv 8pd+r \pmodtight{8p^2}.
$$
Our claim now follows by taking $t=r+p$.
We further write $t=c+\varepsilon p$ with $0\leqslant c<p$ and $1\leqslant \varepsilon \leqslant 8$. It is clear that the representation is unique.
\end{proof}

\begin{lemma}\label{lemma36}
Let $A$ be defined in \eqref{defineA}. Then we have
\[
  A+A=\mathbb Z_{8p^2}.
\]
\end{lemma}
\begin{proof}
Choose $n\in \mathbb{Z}_{8p^2}$ and we may assume that  $0\leqslant n<8p^2$.  By Lemma \ref{lem-rep}, we express $n$ in the form
\begin{align*}n\equiv c+\varepsilon  p+8pd\pmodtight{8p^2},
  \qquad 0\leqslant c<p,
  \qquad 1\leqslant \varepsilon \leqslant 8, \qquad 0\leqslant d<p,
\end{align*}

If $d-mc^2$ is a non-zero quadratic residue modulo $p$, we choose $(i,j)=(1,2)\in \mathbb{N}^2$, and otherwise we choose $(i,j)=(3,3)$. By Lemma \ref{lemma32}, 
there exist $(u_1,v_1)\in \mathcal{Q}_{k_i}$ and $(u_2,v_2)\in \mathcal{Q}_{k_j}$ such that
\[
  (c,d)=(u_1,v_1)+(u_2,v_2)\in \mathbb{Z}_p^2.
\]
We may assume that $u_1,u_2\in\{0,\ldots,p-1\}$. Since $u_1+u_2\equiv c\pmodtight{p}$, there exists a unique $q\in\{0,1\}$ (depending on $u_1,u_2$ and $c$) such that
\begin{align}
  u_1+u_2=c+pq.
  \label{define-q}
\end{align}

We conclude from  $1\leqslant \varepsilon \leqslant 8$ and  $q\in\{0,1\}$ that $\varepsilon -q\in \{0,\ldots,8\}$. In view of \eqref{Ecomplete}, (no matter 
$(i,j)=(1,2)$ or $(3,3)$) there exist $e_1\in E_{i}$ and $e_2\in E_j$ such that
\[
  e_1+e_2=\varepsilon -q.
\]

Recalling \eqref{defineAi}, we have $u_1+pe_1+8pv_1\in A_i$ and $u_2+pe_2+8pv_2\in A_j$.  Now we can deduce that
\begin{align*}
  (u_1+pe_1+8pv_1)+(u_2+pe_2+8pv_2)= &\ c+pq+p(\varepsilon -q)+8p(v_1+v_2)
  \\ =&\ c+p\varepsilon +8p(v_1+v_2)
\end{align*}
and therefore
\begin{align*}
  (u_1+pe_1+8pv_1)+(u_2+pe_2+8pv_2)\equiv c+p\varepsilon +8pd=n\pmodtight{8p^2}.
\end{align*}
This yields $n\in A_i+A_j$ and thus $n\in A+A$. This completes the proof of this lemma.
\end{proof}

For $i,j\in\{1,2,3\}$ and an integer $s$, we define 
\begin{align}\label{define-rho}
  \rho_{ij}(s)=\#\{(e_1,e_2)\in E_i\times E_j:\ e_1+e_2=s\}.
\end{align}
For $1\leqslant \varepsilon \leqslant 8$ and $h\in \{-1,0,1,2\}$, we define
\begin{align}\label{defineomega}
  \omega^{\varepsilon ,h}_{ij}=\max_{q\in\{0,1\}} \rho_{ij}(\varepsilon +8h-q).
\end{align}
With the above notation, we can now introduce the following lemma.

\begin{lemma}\label{lemmaboundsigma}Let $A$ be defined in \eqref{defineA}. Let $0\leqslant n<8p^2$ and (by Lemma \ref{lem-rep})  we express $n$ (uniquely) in the form
\begin{align}\label{represent-n}n\equiv c+\varepsilon  p+8pd\pmodtight{8p^2},
  \qquad 0\leqslant c<p,
  \qquad 1\leqslant \varepsilon \leqslant 8, \qquad 0\leqslant d<p.
\end{align}Then we have
\begin{align*}
  \sigma_A(n)
  \leqslant \sum_{h=-1}^{2}\Big(2\omega^{\varepsilon ,h}_{11}+4\omega^{\varepsilon ,h}_{13}+2\omega^{\varepsilon ,h}_{22}+4\omega^{\varepsilon ,h}_{23}+\max\{4\omega^{\varepsilon ,h}_{12},\ 2\omega^{\varepsilon ,h}_{33},\ 2\omega^{\varepsilon ,h}_{12}+\omega^{\varepsilon ,h}_{33}\}\Big).
\end{align*}
\end{lemma}
\begin{proof}It is clear that 
\begin{align}\label{boundsigma}\sigma_A(n)\leqslant \sum_{i=1}^3\sum_{j=1}^{3}\sigma_{A_i,A_j}(n).
\end{align}
Now we fix $i,j$ and proceed to bound $\sigma_{A_i,A_j}(n)$. Suppose that $x\in A_i$ and $y\in A_j$ satisfy
\begin{align}\label{xyn}x+y\equiv n\pmodtight{8p^2}.\end{align} Recalling the definition of $A_i$ in \eqref{defineAi}, we express $x$ and $y$ uniquely in the form
\begin{align}\label{expressxy}
  x=u_1+pe_1+8pv_1,
  \qquad y=u_2+pe_2+8pv_2,
\end{align}
where  
\begin{align}\label{condition}(u_1,v_1)\in \mathcal{Q}_{k_i},\ e_1\in E_i,\ (u_2,v_2)\in \mathcal{Q}_{k_j},\ e_2\in E_j\ \textrm{ and }\ u_1,u_2\in\{0,\ldots,p-1\}.\end{align}
Then we conclude from \eqref{represent-n}, \eqref{xyn}  and \eqref{expressxy} that 
\begin{align}\label{u12}
 u_1+u_2\equiv c\pmodtight{p}.
\end{align}
Thus, we can write $u_1+u_2=c+pq$, where $q\in\{0,1\}$. Now \eqref{xyn} is equivalent to 
\begin{align}\label{before-h}
  q+e_1+e_2-\varepsilon +8(v_1+v_2-d)\equiv 0\pmodtight{8p}.
\end{align}
Then we define $h$ as 
\begin{align}\label{he1e2}
  h=\frac{q+e_1+e_2-\varepsilon }{8}.
\end{align}
It follows from \eqref{before-h} that $h$ is an integer. Then we further conclude from \eqref{before-h} and \eqref{he1e2} that 
\begin{align}\label{v12}
  v_1+v_2\equiv d-h\pmodtight{p}.
\end{align}
Note that 
$$
q+e_1+e_2-\varepsilon \leqslant 1+9+9-1=18
$$ 
and 
$$
q+e_1+e_2-\varepsilon \geqslant 0+(-3)+(-3)-8=-14.
$$
 Then we obtain 
$$h\in \{-1,0,1,2\}.$$

Let $\tau_{ij}(h)$ denote the number of solutions of $(u_1,e_1,v_1,u_2,e_2,v_2)$ satisfying \eqref{condition}, \eqref{u12}, \eqref{he1e2} and \eqref{v12}. Then we conclude from above that 
$$
\sigma_{A_i,A_j}(n)\leqslant \sum_{h=-1}^2\tau_{ij}(h)
$$ 
and
\begin{align*}\tau_{ij}(h)\leqslant r_{ij}(c,d-h)\max_{q\in\{0,1\}} \rho_{ij}(\varepsilon +8h-q).\end{align*}
Recalling \eqref{defineomega}, we arrive at
\begin{align*}\sigma_{A_i,A_j}(n)\leqslant \sum_{h=-1}^2r_{ij}(c,d-h)\omega^{\varepsilon ,h}_{ij}.\end{align*}
Putting the above into \eqref{boundsigma}, we obtain
\begin{align}\label{boundsigma2}\sigma_A(n)\leqslant \sum_{h=-1}^2\sum_{i=1}^3\sum_{j=1}^{3}r_{ij}(c,d-h)\omega^{\varepsilon ,h}_{ij}.
\end{align}

For $(i,j)=(1,1),(1,3), (3,1), (2,2), (2,3), (3,2)$, we conclude by Lemma \ref{lemma31} that 
$$
r_{ij}(c,d-h)\leqslant 2.
$$ 
For the remaining triples $(i,j)=(1,2),(2,1),(3,3)$, Lemma \ref{lemma32} gives, for each fixed $h$,
$$
\sum_{(i,j)=(1,2),(2,1), (3,3)}r_{ij}(c,d-h)\omega^{\varepsilon ,h}_{ij}
\leqslant \max\big\{4\omega^{\varepsilon ,h}_{12},\ 2\omega^{\varepsilon ,h}_{33},\ 2\omega^{\varepsilon ,h}_{12}+\omega^{\varepsilon ,h}_{33}\big\},
$$
where we have used the observation that $\omega^{\varepsilon ,h}_{ij}=\omega^{\varepsilon ,h}_{ji}$. Putting the above into \eqref{boundsigma2}, we obtain
\begin{align*}
\sigma_A(n)\leqslant
\sum_{h=-1}^{2}\Big(&2\omega^{\varepsilon ,h}_{11}+4\omega^{\varepsilon ,h}_{13}+2\omega^{\varepsilon ,h}_{22}+4\omega^{\varepsilon ,h}_{23}+\max\{4\omega^{\varepsilon ,h}_{12},\ 2\omega^{\varepsilon ,h}_{33},\ 2\omega^{\varepsilon ,h}_{12}+\omega^{\varepsilon ,h}_{33}\}\Big),
\end{align*}
which is the desired bound.
\end{proof}

\begin{proof}[Proof of Theorem \ref{thm2}]
For $\varepsilon $ and $h$, we introduce
\begin{align*}
  W^{\varepsilon ,h}=2\omega^{\varepsilon ,h}_{11}+4\omega^{\varepsilon ,h}_{13}+2\omega^{\varepsilon ,h}_{22}+4\omega^{\varepsilon ,h}_{23}+\max\{4\omega^{\varepsilon ,h}_{12},\ 2\omega^{\varepsilon ,h}_{33},\ 2\omega^{\varepsilon ,h}_{12}+\omega^{\varepsilon ,h}_{33}\}
\end{align*}
and
$$Y^{\varepsilon }=\sum_{h=-1}^2W^{\varepsilon ,h}.$$
Note that $\varepsilon $ is uniquely determined by $n$. 
By Lemma \ref{lemmaboundsigma}, we have 
$$R_{8p^2} \leqslant \sup_{\varepsilon }Y^{\varepsilon }.$$
For fixed $\varepsilon $ and $h$, in view of \eqref{defineE}, \eqref{define-rho} and \eqref{defineomega}, we can compute all $\omega^{\varepsilon ,h}_{ij}~(1\leqslant i,j\leqslant 3)$, and thus obtain the values of $W^{\varepsilon ,h}$ and $Y^{\varepsilon }$. We list all values in Tables~\ref{tab:omega} and~\ref{tab:Y}. These finite computations are also verified by the code in Appendix~\ref{app:verification}.

From Table~\ref{tab:Y} we have $\sup_{\varepsilon }Y^{\varepsilon}=32$, and this establishes $R_{8p^2}\leqslant 32$ for every prime $p\geqslant 7$. 
For $p\leqslant 5$, we use a simple direct construction. Let
$$
A=\{0,1,\ldots,2p-1\}\cup\{2pt:0\leqslant t\leqslant 4p-1\}\subseteq \mathbb Z_{8p^2}.
$$
Then $|A|=2p+4p-1=6p-1$. Moreover, every residue class modulo $8p^2$ can be written uniquely in the form
$$
n\equiv r+2pt\pmodtight{8p^2},\qquad 0\leqslant r<2p,
\qquad 0\leqslant t<4p.
$$
Here $r\in A$ and $2pt\in A$, so $A+A=\mathbb Z_{8p^2}$. Also, for each fixed $x\in A$ there is at most one $y\in A$ with $x+y\equiv n\pmodtight{8p^2}$; hence $\sigma_A(n)\leqslant |A|=6p-1\leqslant 29<32$. 
  
 This completes the proof of Theorem \ref{thm2}.\end{proof}

\begingroup
\footnotesize
\setlength{\tabcolsep}{3.5pt}
\renewcommand{\arraystretch}{1.08}
\begin{longtable}{ccccccccc}
\caption{Values of $\omega_{ij}^{\varepsilon ,h}$ and $W^{\varepsilon ,h}$}\label{tab:omega}\\
\toprule
$\varepsilon $&$h$&$\omega^{\varepsilon ,h}_{11}$&$\omega^{\varepsilon ,h}_{12}$&$\omega^{\varepsilon ,h}_{13}$&$\omega^{\varepsilon ,h}_{22}$&$\omega^{\varepsilon ,h}_{23}$&$\omega^{\varepsilon ,h}_{33}$&$W^{\varepsilon ,h}$\\
\midrule
\endfirsthead
\caption[]{Values of $\omega_{ij}^{\varepsilon ,h}$ and $W^{\varepsilon ,h}$ continued}\\
\toprule
$\varepsilon $&$h$&$\omega^{\varepsilon ,h}_{11}$&$\omega^{\varepsilon ,h}_{12}$&$\omega^{\varepsilon ,h}_{13}$&$\omega^{\varepsilon ,h}_{22}$&$\omega^{\varepsilon ,h}_{23}$&$\omega^{\varepsilon ,h}_{33}$&$W^{\varepsilon ,h}$\\
\midrule
\endhead
\midrule
\multicolumn{9}{r}{\emph{Continued on the next page}}\\
\endfoot
\bottomrule
\endlastfoot
1&-1&0&0&0&0&0&0&0\\
1&0&0&1&2&0&0&2&12\\
1&1&0&1&0&2&2&1&16\\
1&2&0&0&0&0&0&0&0\\
2&-1&1&0&0&0&0&0&2\\
2&0&0&1&2&0&0&2&12\\
2&1&0&0&0&2&2&0&12\\
2&2&0&0&0&1&0&0&2\\
3&-1&2&0&0&0&0&0&4\\
3&0&0&1&2&0&1&2&16\\
3&1&0&0&0&0&2&0&8\\
3&2&0&0&0&1&0&0&2\\
4&-1&3&0&0&0&0&0&6\\
4&0&0&1&1&0&1&4&16\\
4&1&0&0&0&3&1&0&10\\
4&2&0&0&0&0&0&0&0\\
5&-1&3&0&1&0&0&0&10\\
5&0&0&1&0&0&1&4&12\\
5&1&0&0&0&3&1&0&10\\
5&2&0&0&0&0&0&0&0\\
6&-1&2&0&2&0&0&0&12\\
6&0&0&1&0&1&2&2&14\\
6&1&0&0&0&0&1&0&4\\
6&2&0&0&0&0&0&0&0\\
7&-1&1&0&2&0&0&0&10\\
7&0&0&1&0&1&2&2&14\\
7&1&0&0&0&2&0&0&4\\
7&2&0&0&0&0&0&0&0\\
8&-1&0&1&2&0&0&1&12\\
8&0&0&1&0&0&2&2&12\\
8&1&0&0&0&2&0&0&4\\
8&2&0&0&0&0&0&0&0\\
\end{longtable}
\endgroup

\begingroup
\footnotesize
\setlength{\tabcolsep}{5pt}
\renewcommand{\arraystretch}{1.08}
\begin{longtable}{cccccc}
\caption{Values of $W^{\varepsilon ,h}$ and $Y^{\varepsilon }$}\label{tab:Y}\\
\toprule
$\varepsilon $&$W^{\varepsilon ,-1}$&$W^{\varepsilon ,0}$&$W^{\varepsilon ,1}$&$W^{\varepsilon ,2}$& $Y^{\varepsilon }$\\
\midrule
\endfirsthead
\caption[]{Values of $W^{\varepsilon ,h}$ and $Y^{\varepsilon }$ continued}\\
\toprule
$\varepsilon $&$W^{\varepsilon ,-1}$&$W^{\varepsilon ,0}$&$W^{\varepsilon ,1}$&$W^{\varepsilon ,2}$& $Y^{\varepsilon }$\\
\midrule
\endhead
\midrule
\multicolumn{6}{r}{\emph{Continued on the next page}}\\
\endfoot
\bottomrule
\endlastfoot
1&0&12&16&0&28\\
2&2&12&12&2&28\\
3&4&16&8&2&30\\
4&6&16&10&0&32\\
5&10&12&10&0&32\\
6&12&14&4&0&30\\
7&10&14&4&0&28\\
8&12&12&4&0&28\\
\end{longtable}
\endgroup

\section*{Acknowledgments}
The author acknowledges the use of OpenAI’s ChatGPT during the preparation of this manuscript. This work is supported by the National Key Research and Development Program of China (Grant No. 2021YFA1000700) and the National Natural Science Foundation of China (Grant No. 12471088).

\appendix
\section{Finite verification}\label{app:verification}
The two finite computations used in the paper are independent of any unbounded parameter. For completeness and reproducibility, we include a short Python script which verifies Lemma~\ref{lem1} and Tables~\ref{tab:omega}--\ref{tab:Y}.
\small
\begin{verbatim}
from collections import Counter

# Verification of Lemma 2.1.
c = [74*j - (j*j)//32 for j in range(265)]
sums = Counter(ci + cj for ci in c for cj in c)
assert max(sums.values()) == 17

# Verification of Tables 1 and 2.
E = {
    1: [-3, -2, -1],
    2: [3, 6, 9],
    3: [0, 1, 3, 4],
}

def rho(i, j, s):
    return sum(1 for e1 in E[i] for e2 in E[j] if e1 + e2 == s)

def omega(i, j, eps, h):
    return max(rho(i, j, eps + 8*h - q) for q in (0, 1))

def W(eps, h):
    o = lambda i, j: omega(i, j, eps, h)
    return (2*o(1,1) + 4*o(1,3) + 2*o(2,2) + 4*o(2,3)
            + max(4*o(1,2), 2*o(3,3), 2*o(1,2) + o(3,3)))

Y = {eps: sum(W(eps, h) for h in (-1, 0, 1, 2))
     for eps in range(1, 9)}
assert Y == {1: 28, 2: 28, 3: 30, 4: 32,
             5: 32, 6: 30, 7: 28, 8: 28}
assert max(Y.values()) == 32
\end{verbatim}

\end{document}